\documentclass{amsart}
\usepackage[utf8]{inputenc}
\usepackage{amsmath}
\usepackage{amssymb}
\usepackage{caption}
\usepackage{amsthm}
\usepackage[hidelinks]{hyperref}
\usepackage{cleveref}
\usepackage{hyperref}
\usepackage{nicefrac}
\usepackage[inline]{enumitem}
\usepackage{enumitem}
\usepackage[bb = fourier,cal = euler,scr = rsfs]{mathalfa}	
\usepackage{mathabx}

\usepackage{amssymb}
\usepackage{mathrsfs}

\usepackage{array,float}

\usepackage{tikz}
\usetikzlibrary{shapes,arrows}
\usetikzlibrary{fit,positioning}
\usetikzlibrary{patterns,decorations.pathreplacing}
\usepackage{xkeyval}
\usepackage{moreverb}
\usepackage{epic}
\usepgfmodule{shapes,plot,decorations}

\usepackage{booktabs} 
\usepackage[normalem]{ulem}

\usepackage[warn]{mathtext}

\newtheorem{theorem}{Theorem}[section]
\newtheorem{proposition}[theorem]{Proposition}
\newtheorem{corollary}[theorem]{Corollary}

\theoremstyle{remark}
\newtheorem{remark}[theorem]{Remark}

\theoremstyle{remark}

\title{On transversality in flag manifolds and linearity of amalgams}

\author{Sami Douba}
\address{Mathematisches Institut der Universit\"at Bonn, Endenicher Allee 60, 53115 Bonn, Germany}
\email{douba@math.uni-bonn.de}

\author{Konstantinos Tsouvalas}
\address{Max Planck Institute for Mathematics in the Sciences, Inselstraße 22, 04103 Leipzig, Germany}
\email{konstantinos.tsouvalas@mis.mpg.de}

\usepackage{chngcntr}
\usepackage{graphicx} 
\usepackage{float}
\counterwithout{equation}{section}
\counterwithout{theorem}{section}

\begin{document}

\begin{abstract}
We show that the fundamental group of the double of a complete negatively curved locally symmetric manifold along a closed geodesic is linear. More generally, we establish linearity of doubles of torsion-free {\em transverse} subgroups (also known in the literature as {\em regular antipodal} subgroups) of semisimple Lie groups along biproximal maximal cyclic subgroups.
\end{abstract}


\maketitle

Given a subgroup $\Delta$ of a group $\Gamma$, we will refer to the amalgam $\Gamma \ast_\Delta \Gamma$ induced by the inclusion into either factor as the {\em double of $\Gamma$ along $\Delta$}. If $\Gamma$ is the fundamental group of a connected manifold $M$ and $\Delta$ the fundamental group of a $\pi_1$-injective properly immersed connected submanifold $N$ of $M$, then the double $\Gamma \ast_\Delta \Gamma$ is nothing but the fundamental group of the {\em double of $M$ along $N$}, where the latter is analogously defined as a graph of spaces. In this note, we prove the following.

\begin{theorem}\label{lss}
The fundamental group of the double of a negatively curved locally symmetric manifold\footnote{All locally symmetric spaces in this note are assumed complete, though not necessarily of finite volume.} $M$ along a primitive closed geodesic in $M$ admits a faithful linear representation over $\mathbb{R}$ whose dimension depends only on that of~$M$.
\end{theorem}

The negative curvature assumption on $M$ in Theorem~\ref{lss} is indispensable. Indeed, as shown in~\cite{TT-IMRN}, it follows from Margulis's superrigidity theorem \cite{Mar} and work of Prasad--Rapinchuk \cite{MR1960120} that if $M$ is instead taken to be an irreducible finite-volume quotient of a {\em higher-rank} symmetric space of noncompact type, then the fundamental group of the double of $M$ along a generic primitive closed geodesic will fail to admit a faithful finite-dimensional linear representation over any field.
On the other hand, as already noted in~\cite[\S3]{MR2100678} in the arithmetic case, for any locally symmetric space $M$ of noncompact type with finitely generated fundamental group, the fundamental group of the double of $M$ along a closed geodesic, primitive or otherwise, is at the very least residually finite; see~\cite{arXiv:2504.19995} and the references~\cite{MR1769939, MR44570, zbMATH03678091} therein.

In the case that $M$ has {\em constant} negative curvature, i.e., the case that $M$ is real hyperbolic, Theorem~\ref{lss} follows from~\cite[Thm.~1.7]{arXiv:2603.23969}.\footnote{Theorem 1.7 in~\cite{arXiv:2603.23969} ceases to apply when $M$ is of variable curvature; indeed, any irreducible symmetric space containing a one-dimensional reflective submanifold must also contain a totally geodesic hypersurface and must thus be of constant curvature \cite{zbMATH03236701, zbMATH03719950}.} In this setting, it moreover follows from work of Baker and Cooper~\cite{MR2417445} that if $c$ is a primitive closed geodesic in~$M$ with trivial holonomy, then there is a finite cover $\hat{M}$ of $M$ to which $c$ lifts isometrically and such that the double of $\hat{M}$ along $c$ is homotopy equivalent to a real hyperbolic manifold whose dimension depends only on that of~$M$. 

Note however that, since the octonionic hyperbolic plane does not embed as a positive-codimension totally geodesic submanifold of any negatively curved symmetric space~\cite{zbMATH03186462}, it follows from the superrigidity theorem of Corlette~\cite{Corlette} that the double of a finite-volume octonionic hyperbolic surface along a properly immersed totally geodesic submanifold (of positive codimension) will fail to be homotopy equivalent to any negatively curved locally symmetric space. It was shown in~\cite{TT-IMRN} that the latter also holds for doubles of compact quaternionic hyperbolic $(\geq 3)$-folds along closed geodesics. By a similar argument, the same will hold for the double of a finite-volume quaternionic hyperbolic $(\geq3)$-fold along a closed geodesic $c$ if $c$ is taken within a properly immersed totally geodesic quaternionic curve.\footnote{In light of Gromov--Schoen's arithmeticity theorem \cite{GS}, such a quaternionic curve is always available; see~\cite[Prop.~4.1]{zbMATH07729896}.}

On the other hand, by the combination theorems of Bestvina--Feighn~\cite{BF} and Dahmani~\cite{zbMATH02057405}, the fundamental group of any double $D$ of a locally symmetric space~$M$ as in the previous paragraph along a primitive closed geodesic $c$ will nevertheless be intrinsically negatively curved, in the sense that such $\pi_1(D)$ are hyperbolic relative to the fundamental groups of the ends of~$D$. If $c$ is moreover chosen so that the (pushforward of the) normalized length measure on $c$ closely approximates the normalized volume measure on $M$, then $\pi_1(D)$ is thus perhaps a good candidate for a linear group that is hyperbolic relative to finitely generated nilpotent subgroups but admits no discrete embedding in any connected Lie group. Note however that, by work of the second-named author with Dey~\cite{arXiv:2504.21802}, in the case that $M$ is compact, up to replacing $M$ with a finite cover, one can always find a representation of~$\pi_1(D)$ that is {\em Anosov} in the sense of Labourie~\cite{MR2221137} and Guichard--Wienhard~\cite{MR2981818}, and is in particular discrete and faithful. As suggested to the second-named author by Hee~Oh, more generally, it seems likely that if $M$ is only assumed to be of finite volume then, up to replacing $M$ with a finite cover, one can always find a representation of $\pi_1(D)$ that is, as defined in \cite{KL-relative} and \cite{Zhu-dominated}, {\em Anosov relative to} the fundamental groups of the ends of $D$.

Theorem~\ref{lss} was previously established in work of the second-named author with Tholozan~\cite{TT-IMRN} in the case that $M$ is convex cocompact (e.g., the case that $M$ is compact). There, the statement followed from a more general linearity result for doubles of torsion-free Anosov subgroups of semisimple Lie groups along maximal cyclic subgroups. The approach here is similar, but Theorem~\ref{lss} will instead follow from a linearity result (Theorem~\ref{doubletransverse}) for doubles of {\em transverse subgroups} in the language of Canary--Zhang--Zimmer~\cite{CZZ}.

To define the latter notion, we introduce some terminology/notation. 
For $n\geq 2$, let~$\mathcal{F}_{1,n-1}$ be the space of all flags in $\mathbb{R}^n$ of the form $(V,W)$, where $V$ and $W$ are linear subspaces of $\mathbb{R}^n$ of dimension $1$ and $n-1$, respectively (for $n=2$, one can understand $\mathcal{F}_{1,n-1}$ simply as $\mathbb{RP}^1$). Two flags $(V,W), (V',W') \in \mathcal{F}_{1,n-1}$ are {\em transverse} if both $V+W' = \mathbb{R}^n$ and $V'+W = \mathbb{R}^n$. A subset $\Lambda \subset \mathcal{F}_{1,n-1}$ is then {\em transverse} if any two distinct elements of $\Lambda$ are transverse.
A sequence $(\gamma_m)_{m \in \mathbb{N}}$ in~$\Gamma$ is {\em \textup{(}$P_{1,n-1}$-\textup{)}contracting} if there are flags $\phi^+, \phi^- \in \mathcal{F}_{1,n-1}$, the {\em attracting} and {\em repelling} flags of $(\gamma_m)_m$, respectively, such that $\gamma_m$ converges to the constant function~$\phi^+$ uniformly on compact subsets of $\mathcal{F}_{1,n-1}$ consisting entirely of flags transverse to~$\phi^-$. A subgroup $\Gamma$ of $\mathsf{SL}_n(\mathbb{R})$ is said to be {\em \textup{(}$P_1$-\textup{)}divergent} if each infinite sequence of distinct elements of $\Gamma$ contains a contracting subsequence. The {\em limit set} $\Lambda_\Gamma$ of such~$\Gamma$ in $\mathcal{F}_{1,n-1}$ is the closed $\Gamma$-invariant subset of $\mathcal{F}_{1,n-1}$ consisting  of all attracting (equivalently, repelling) flags of contracting sequences in $\Gamma$. Note that a divergent subgroup of $\mathsf{SL}_n(\mathbb{R})$ is in particular discrete, but that, as soon as $n \geq 3$, discrete subgroups of $\mathsf{SL}_n(\mathbb{R})$ need not be divergent (for example, if $n \geq 3$, then no lattice in a Cartan subgroup of, and hence by~\cite{MR302822} no lattice in, $\mathsf{SL}_n(\mathbb{R})$ is divergent). One says a divergent subgroup $\Gamma$ of $\mathsf{SL}_n(\mathbb{R})$ is {\em \textup{(}$P_1$-\textup{)}transverse} if $\Lambda_\Gamma$ is transverse. In the latter case, the action of $\Gamma$ on $\Lambda_\Gamma$ is a convergence action in the sense of Gehring and Martin \cite{GM}.\footnote{In retrospect, the notion of a (nonelementary) convergence action was introduced earlier by Furstenberg~\cite{zbMATH03293802}; see~\cite{zbMATH05166054}.} An element $\gamma \in \mathsf{SL}_n(\mathbb{R})$ generating a transverse cyclic subgroup whose limit set consists of precisely two flags $\gamma^\pm \in \mathcal{F}_{1,n-1}$ is said to be {\em biproximal}. In the language of convergence actions, an infinite-order element $\gamma$ of a transverse subgroup $\Gamma < \mathsf{SL}_n(\mathbb{R})$ is biproximal if and only if $\gamma$ acts loxodromically on $\Lambda_\Gamma$. 

The main result of this note is the following.

\begin{theorem}\label{doubletransverse}
Suppose $\Gamma < \mathsf{SL}_n(\mathbb{R})$ is transverse and $\gamma \in \Gamma$ is biproximal. Let $\Delta = \mathrm{Stab}_\Gamma(\gamma^+)$. Then the double $\Gamma \ast_\Delta \Gamma$ embeds in $\mathsf{SL}_n(\mathbb{R})$. 
\end{theorem}

\begin{remark}
In the context of Theorem \ref{doubletransverse}, we also have $\Delta = \mathrm{Stab}_\Gamma(\gamma^-)$, and $\Delta$ is virtually cyclic. In particular, if $\Gamma$ is moreover torsion-free, then $\Delta$ is nothing but the maximal cyclic subgroup of $\Gamma$ containing $\gamma$. 
\end{remark}

The above notions of transversality, contraction, divergence, limit set, and biproximality can be (and have been) extended to general reflexive flag manifolds $G/P$ of semisimple  Lie groups $G$; see Kapovich--Leeb--Porti~\cite{KLP}, whose work precedes that of Canary--Zhang--Zimmer~\cite{CZZ} (though we adopt here the terminology of the latter), and who instead refer to such subgroups as {\em regular antipodal}. Examples of $P$-transverse subgroups of such $G$ include the relative~$P$-Anosov subgroups~\cite{KL-relative, Zhu-dominated}. One can realize $P$-transverse subgroups of $G$ for general pairs $(G,P)$ as transverse subgroups of special linear groups via the following fact.

\begin{proposition}\label{plucker}\textup{(}\cite[Prop.~B.1]{CZZ}\textup{)}
Let $G$ be a noncompact semisimple real algebraic group and $P < G$ a reflexive proper parabolic subgroup (considered up to conjugacy). Then there exist $n = n(G,P) \in \mathbb{N}$, a representation $\tau: G \rightarrow \mathsf{SL}_n(\mathbb{R})$, and a $\tau$-equivariant smooth embedding $\xi: G/P \rightarrow \mathcal{F}_{1,n-1}$, such that a given subgroup $\Gamma < G$ is $P$-transverse if and only if $\tau(\Gamma)$ is transverse in $\mathsf{SL}_n(\mathbb{R})$, and in the latter case, one has $\Lambda_{\tau(\Gamma)} = \xi(\Lambda_{\Gamma})$, where $\Lambda_{\Gamma}$ denotes the limit set of $\Gamma$ in $G/P$. In particular, an element $g \in G$ is $P$-biproximal if and only if $\tau(g)$ is biproximal.
\end{proposition}

The representation $\tau$ in Proposition~\ref{plucker} is nothing but a so-called Pl\"ucker--Tits representation of~$G$; see, for instance,~\cite[Prop. 3.3~and~Lem. 3.7]{GGKW}. The following is then immediate from Theorem~\ref{doubletransverse} and Proposition~\ref{plucker}.

\begin{corollary}\label{doublePtransverse}
Let $G$ be a noncompact semisimple real algebraic group and $P < G$ a reflexive proper parabolic subgroup (considered up to conjugacy). Then there exists $n = n(G,P) \in \mathbb{N}$ such that, for any torsion-free $P$-transverse subgroup $\Gamma < G$ and any maximal cyclic subgroup $\Delta < \Gamma$ generated by a $P$-biproximal element of $G$, the double $\Gamma \ast_\Delta \Gamma$ embeds in $\mathsf{SL}_n(\mathbb{R})$.
\end{corollary}

Note that, since $G$ contains only finitely many conjugacy classes of proper parabolic subgroups, it follows formally that the dimension $n$ can in fact be made uniform in $P$.

If $G$ is a simple Lie group of real rank one and $P$ is any proper parabolic subgroup of $G$, then $G/P$ can be identified, as a homogeneous space for $G$, with the visual boundary of the symmetric space for $G$. In this setting, to say two elements of $G/P$ are transverse is merely to say that they are distinct; a $P$-transverse subgroup~$\Gamma$ of $G$ is nothing but a discrete subgroup of $G$; and the above notion of limit set in~$G/P$ coincides with the classical one. Theorem~\ref{lss} then follows immediately from Corollary~\ref{doublePtransverse}.

Theorem~\ref{doubletransverse} will in turn be deduced from the following statement, which is inspired by techniques of Shalen~\cite{Shalen}.

\begin{theorem}\label{amalgam}
Let $\Gamma_1, \Gamma_2 < \mathsf{SL}_n(K)$ and $\Delta < \Gamma_1 \cap \Gamma_2$, where $K$ is any field. Suppose there is an integer $1 \leq m \leq \frac{n}{2}$ such that for each $h \in \Delta$, we have that $h$ is of the form
\[h= \begin{pmatrix} A_h & & \\ & B_h & \\ & & C_h \end{pmatrix}
\]
for some $A_h, C_h \in \mathsf{GL}_m(K)$ and $B_h \in \mathsf{GL}_{n-2m}(K)$. Suppose moreover that for each $g_1 \in \Gamma_1 \smallsetminus \Delta$ the bottom-left $(m \times m)$-block of $g_1$ is invertible, and for each $g_2 \in \Gamma_2 \smallsetminus \Delta$, the top-right $(m \times m)$-block of $g_2$ is invertible. If $L$ is a field extension of $K$ and $t \in L$ is transcendental over $K$, then the subgroup $\langle \Gamma_1, a_t \Gamma_2 a_t^{-1} \rangle < \mathsf{SL}_n(L)$ decomposes as the amalgam $\Gamma_1 *_\Delta a_t \Gamma_2 a_t^{-1}$, where
\[
a_t := \begin{pmatrix} tI_m & & \\ & I_{n-2m} & \\ & & t^{-1}I_m \end{pmatrix}.
\]
\end{theorem}

\begin{proof}
We show that no element $g \in \mathsf{SL}_n(L)$ of the form  
\begin{equation*}
g = g_1^{(1)}g_2^{(1)}\ldots g_1^{(r)}g_2^{(r)}
\end{equation*}
is unipotent, where $g_1^{(1)}, \ldots, g_1^{(r)} \in \Gamma_1 \smallsetminus \Delta$ and $g_2^{(1)}, \ldots, g_2^{(r)} \in a_t\Gamma_2a_t^{-1} \smallsetminus \Delta$. This will suffice since any element of $\Gamma_1 *_\Delta  \Gamma_2$, if not conjugate into one of the factors, is conjugate to an alternating product of even length.

Given a matrix $M  \in \mathsf{M}_n(L)$, we write
\[
M = \begin{pmatrix} M_{11} & M_{12} \\ M_{21} & M_{22} \end{pmatrix},
\]
where the $M_{ij}$ are matrices with entries in $L$, and $M_{22} \in \mathsf{M}_m(L)$. We say $M  \in \mathsf{M}_n(L)$ is {\it sufficiently transcendental} if there is a nonnegative integer $q$ such that
\begin{itemize}
\item $M_{11} = \ldots  + t^{-1}W_{-1} + W_0 + tW_1 + \ldots + t^q W_q$ for some $W_i \in \mathsf{M}_{n-m}(K)$;
\item $M_{12} = \ldots + t^{-1}X_{-1} + X_0 + tX_1 + \ldots + t^{q+1} X_{q+1}$  for some $X_i \in \mathsf{M}_{(n-m)\times m}(K)$;
\item $M_{21} = \ldots + t^{-1}Y_{-1} + Y_0 + tY_1 + \ldots + t^q Y_{q}$ for some $Y_i \in \mathsf{M}_{m \times (n-m)}(K)$;
\item $M_{22} = \ldots + t^{-1}Z_{-1} + Z_0 + tZ_1 + \ldots + t^{q+1} Z_{q+1}$ for some $Z_i \in \mathsf{M}_{m}(K)$ with $Z_{q+1}$ invertible.
\end{itemize}
Observe that since $t$ is transcendental over $K$, the zero matrix is not sufficiently transcendental. Note also that the sum of a sufficiently transcendental matrix with an element of $\mathsf{M}_n(K)$ is sufficiently transcendental, and that a product of sufficiently transcendental matrices is sufficiently transcendental. In particular, by our first observation, a sufficiently transcendental matrix is not nilpotent. It follows that a sufficiently transcendental matrix $M$ is not unipotent, since $M- \mathrm{Id}$ is sufficiently transcendental and therefore not nilpotent.

It is easy to verify that $g_1g_2$ is sufficiently transcendental for $g_1 \in  \Gamma_1 \smallsetminus \Delta$ and $g_2 \in a_t\Gamma_2a_t^{-1} \smallsetminus \Delta$, so that $g$ is sufficiently transcendental and therefore not unipotent by the previous discussion.
\end{proof}

\begin{proof}[Proof~of~Theorem~\ref{doubletransverse}]
Up to conjugating $\Gamma$ within $\mathsf{SL}_n(\mathbb{R})$, we have
\begin{alignat*}{2}\gamma^+ &= (e_1, \textup{span}_{\mathbb{R}}\{e_1, \ldots, e_{n-1} \}), \\ \gamma^- &= (e_n,\textup{span}_{\mathbb{R}}\{ e_2, \ldots, e_n \}),
\end{alignat*}
where $e_1, \ldots, e_n$ denote the standard basis vectors of $\mathbb{R}^n$. As $\Delta = \mathrm{Stab}_\Gamma(\gamma^-)$, we then have that any $h \in \Delta$ is of the form
\[
h = \begin{pmatrix} a & & \\ & B & \\ & & c\end{pmatrix}
\]
for some $a, c \in \mathbb{R}^*$ and $B \in \mathsf{GL}_{n-2}(\mathbb{R})$. Since $\Gamma$ is discrete in $\mathsf{SL}_n(\mathbb{R})$ and hence countable, the entry field $K$ of $\Gamma$ is countable, so that $\mathbb{R}$ is a transcendental extension of $K$. By Theorem~\ref{amalgam} (with $m=1$), it thus suffices to show that, given $g \in \Gamma \smallsetminus \Delta$, both the bottom-left and top-right entries of $g$ are nonzero. Indeed, since $g\gamma^+ \in \Lambda_\Gamma \smallsetminus \{ \gamma^+ \}$, and since~$\Lambda_\Gamma$ is transverse by assumption, we have that $\gamma^+$ and $g \gamma^+$ are transverse, and hence that $g e_1 \notin \langle e_1, \ldots, e_{n-1} \rangle$. In other words, the bottom-left entry of $g$ is nonzero. Since we also have $\Delta = \mathrm{Stab}_\Gamma(\gamma^-)$, an identical argument shows that the top-right entry of $g$ is likewise nonzero.
\end{proof}

\subsection*{Acknowledgements} This note is intended as a contribution to the proceedings of the 40th Summer Conference on Topology and Its Applications, where SD spoke about the related work~\cite{arXiv:2603.23969}. SD thanks the organizers of the Geometric Group Theory session (Ben Linowitz, Ben McReynolds, and Barry Minemyer) for the opportunity to present the latter work, and the University of Split, Croatia, for its hospitality. SD is also grateful to the Max Planck Institute for Mathematics in the Sciences, Leipzig, during a visit to which this note was drafted.

\bibliographystyle{siam}
\bibliography{biblio.bib}

\end{document}